\documentclass[11pt,a4paper]{article}
\usepackage{amsmath,amssymb,amsthm}
\usepackage{authblk}

\newtheorem{theorem}{Theorem}
\newtheorem{corollary}[theorem]{Corollary}
\newtheorem{remark}[theorem]{Remark}
\newtheorem{problem}{Problem}

\newtheorem{question}[problem]{Question}

\newcommand{\be}{\begin{eqnarray}}
\newcommand{\ee}{\end{eqnarray}}
\newcommand{\beq}{\begin{equation}}
\newcommand{\eeq}{\end{equation}}
\newcommand{\benum}{\begin{enumerate}}
\newcommand{\eenum}{\end{enumerate}}
\newcommand{\bal}{\begin{align*}}
\newcommand{\eal}{\end{align*}}
\newcommand{\ba}{\begin{array}}
\newcommand{\ea}{\end{array}}

\newcommand{\Z}{\ensuremath{\mathbb Z}}
\newcommand{\N}{\ensuremath{\mathbb N}}

\title{Monochromatic paths for the integers}

\author[1]{Jo\~{a}o Guerreiro}
\author[2]{Imre Z. Ruzsa\thanks{Author was supported by ERC--AdG Grant No.321104  and
 Hungarian National Foundation for Scientific
 Research (OTKA), Grants No.109789, 
 and NK104183.}}
\author[3]{Manuel Silva}
\affil[1]{Department of Mathematics, Columbia University,
New York, United States of America}
\affil[2]{Alfr\'{e}d R\'{e}nyi Institute of Mathematics, Hungarian Academy of Sciences, Budapest, Hungary}
\affil[3]{Departamento de Matem\'{a}tica, Universidade Nova de Lisboa, Caparica, Portugal}

\date{}

\begin{document}

\maketitle

\begin{abstract} Recall that van der Waerden's theorem states that any finite coloring
of the naturals has arbitrarily long monochromatic arithmetic sequences.
We explore questions about the set of differences of those sequences.
\end{abstract}

\bigskip

\section{Introduction}

A famous result of van der Waerden asserts that any finite coloring of the positive integers contains arbitrarily long monochromatic arithmetic progressions. In a typical \emph{Ramsey problem}, we try to prove that some particular type of structure cannot be avoided. In this article we will consider arithmetic progressions with common difference in some given set $L$, and sets whose consecutive differences belong to a fixed set.

We call a set $L \subset \N$ an $r-$\emph{ladder} if for any $r-$coloring $\alpha : \mathbb{N} \to [1,r]$ there are arbitrarily long monochromatic arithmetic progressions with common difference $r \in L$. If a set $L \subseteq \N$ is an $r-$ladder for all positive integers $r$ then it is called a \emph{ladder}. The set of positive integers $\N$ is an example of a ladder. The following related important open question has been proposed in \cite{BGL}.

\begin{problem}[open]
Is every $2$-ladder set necessarily an $r$-ladder for every $r>2$?
\end{problem}

It is not even known if the $r$-ladder property implies $(r+1)$-ladder property for any $r \geq 2$.

\bigskip

The following regularity property appears also in \cite{BGL}.

\begin{theorem}
If $L = L_1 \cup L_2$ is a ladder, then either $L_1$ or $L_2$ is a ladder.
If $L$ is a ladder, then $L \cap n \N$ is a ladder for any $n \in \N$.
\end{theorem}

\begin{proof}
To prove the first statement, suppose neither $L_1$ nor $L_2$ is a ladder. Then there is a finite coloring of the positive integers for which there are no monochromatic arithmetic progressions with say $k$ elements with difference in $L_1$, and a similar coloring for the set $L_2$.  The Cartesian product of these two colorings will avoid long monochromatic arithmetic progressions with  difference in $L$.
To show the second statement, assume that for any finite coloring of the positive integers there are arbitrarily long arithmetic progressions with difference in $L$. Further assume that elements in different classes modulo $n$ always have different colors (we can always refine the coloring to make this true), then we will have arbitrarily long arithmetic progressions with difference in $L \cap n \N$.
\end{proof}

\begin{remark}
From the previous theorem we can conclude that $n \N$ is a ladder for any $n \in \N$ and that nonzero classes modulo $n$ are not ladders. Note as well that $L$ is a ladder if and only if the complement of $L$ is not a ladder. In particular, all cofinite sets are ladders and finite sets are not.
\end{remark}

The best way to construct non trivial examples of ladders is using the polynomial van der Waerden theorem.

\begin{theorem}[Polynomial van der Waerden]
Let $p_1,\cdots,p_k$ be polynomials with integer coefficients and no constant term. Then for any $r$-coloring of $\N$ there exist $a$, $d$ such that $a, a+p_1(d),\cdots, a+p_k(d)$ are the same color.
\end{theorem} 

See \cite{BL} and \cite{W}, the latter for an elementary proof.

\begin{corollary}
Let $p$ be a polynomial with integer coefficients and no constant term. Then for any $r$-coloring of $\N$ and any $k \geq 1$ there exists $a$, $d$ such that $a, a+p(d),\cdots, a+k p(d)$ are the same color.
\end{corollary}

\begin{proof}
This follows from the polynomial van der Waerden by choosing $p_i = i p$, $1 \leq i \leq k$.
\end{proof}

In particular, if $p$ is a polynomial with integer coefficients and satisfying the condition $p(0) = 0$ then the set $|p(\N)| = p(\N) \cap \N$  is a ladder. 

\begin{question}
Is there a ladder which does not contain a set of the form $|p(\N)|$ for some polynomial $p$ as above?
\end{question}

\bigskip

\section{An elementary construction of sets which are not ladders}

In this section we construct a family of sets that are not ladders. We shall discuss more properties of this family of sets in the following section.

\smallskip

\begin{theorem} \label{construction}
Let $S = \left\{ m^{2k-1} : k \in \N \right\}$. For $m \geq 5$  the set $S-S$ is not a $3$-ladder.
\end{theorem}

\begin{proof}
Let $f_i(n)$ be the $i$-th digit in the base $m$ representation of $n$.
Consider the following  coloring of the positive integers with 3 colors:
\[ \alpha(n) = \left|\left\{i : f_{2i}(n) = 2\right\}\right| \bmod 3. \]

Let $\left\{ x_1, x_2, \ldots, x_n \right\}$ be an $n$-term monochromatic arithmetic progression with difference $d = m^{2j-1} - m^{2k-1}$.
Suppose that $n > m^2+m+1$. Notice that 
\[ \left\{ \big(f_{2k}(x_s), f_{2k-1}(x_s) \big) : 1 \leq s \leq m^2 \right\} = \big\{0,\ldots,m-1 \big\}^2.\]

Take $1 \leq t \leq m^2$ such that
 \[ \big(f_{2k}(x_{t}), f_{2k-1}(x_{t})\big) = (3,0). \]
We can calculate the values of $\big(f_{2k}(x_{t+i}), f_{2k-1}(x_{t+i})\big)$ for $0\leq i\leq m+1$; in particular,
\begin{align*}
   \big(f_{2k}(x_{t+1}), f_{2k-1}(x_{t+1})\big) &= (2, m-1), \\
 \big(f_{2k}(x_{t+m}), f_{2k-1}(x_{t+m})\big) &= (2,0), \\
\big(f_{2k}(x_{t+m+1}), f_{2k-1}(x_{t+m+1})\big) &= (1,m-1).  \\
\end{align*}
The digits at positions $< 2k-1$ do not change, neither those at positions in $(2k+1, 2j-2)$. The behavior of the
highest digits depends on the occurrence of carries.

If $ f_{2j-1}(x_t)$ is neither 2 nor $m-1$, then we get that  $x_{t}$ and $x_{t+1}$ are of different color.
Similarly, if $ f_{2j}(x_t)$ is neither 2 nor $m-1$, then we get that  $x_{t}$ and $x_{t+m}$ are of different color.
If  $\big(f_{2j}(x_{t}), f_{2j-1}(x_{t})\big) = (2, 2)$ or $(2, m-1)$, we can similarly compare  $x_{t}$ and $x_{t+m+1}$.

Finally, if $\big(f_{2j}(x_{t}), f_{2j-1}(x_{t})\big)= (m-1,m-1)$ or $(m-1,2)$, we can compare  $x_{t+m}$ and $x_{t+m+1}$.

In the first case we have 
\begin{align*}
\big(f_{2j}(x_{t+m}), f_{2j-1}(x_{t+m})\big) &= (0,m-1),\\
\big(f_{2j}(x_{t+m+1}), f_{2j-1}(x_{t+m+1})\big) &= (1,0),
\end{align*}
 while all other base $m$ digits remain unchanged. Therefore, $x_{t+m}$ and $x_{t+m+1}$ have different colours which contradicts the hypothesis.

In the second case, we have 
\begin{align*}
\big(f_{2j}(x_{t+m}), f_{2j-1}(x_{t+m})\big) &= (0,2),\\
\big(f_{2j}(x_{t+m+1}), f_{2j-1}(x_{t+m+1})\big) &= (0,3),
\end{align*}
 while all other base $m$ digits remain unchanged. Therefore, $x_{t+m}$ and $x_{t+m+1}$ have different colours which contradicts the hypothesis.

We conclude that $S-S$ is not $3$-ladder
\end{proof}

\section{On intersective and accessible sets}




We call a set $S \subseteq \N$ \emph{accessible} if for any finite coloring $\alpha: \N \to [1,r]$ there are arbitrarily long
monochromatic sets $A = \{ a_1, a_2, \ldots, a_n\}$ such that $a_{i+1} - a_i \in S$. We will say that $A$ is a monochromatic $S$-sequence. The notion of accessible set is defined for example in \cite{LR2}.

Any ladder set is clearly accessible, and so is the
 difference set $L=A-A$ of any infinite set $A \subseteq \mathbb{N}$, or any set that contains arbitrarily large finite difference sets, \cite[Theorem 10.27]{LR1}. 

A set $S \subseteq \N$ is chromatically intersective if for any finite partition of $$\N=A_1 \cup \cdots \cup A_d,$$ there is $1 \leq i \leq d$ such that $$(A_i-A_i) \cap S \neq \emptyset.$$

It is clear that any accessible set is also chromatically intersective. We now prove the converse.

\begin{theorem}\label{chrom-acc}
Every chromatically intersective set $S \subset \N$ is also accessible.
\end{theorem}

\begin{proof}
We need to prove that for any finite coloring there are arbitrarily long monochromatic sequences whose consecutive differences belong to the set $S$. Suppose this is not true. Let $k \geq 2$ be the maximal length over all monochromatic $S$-difference sequences for some particular $r$-coloring $\alpha: \N \to [1,r]$. We may assume that this $k$ is minimal over all finite colorings. Given any two monochromatic maximal $S$-difference sequences $x_1,x_2,\cdots,x_k$ and $y_1,y_2,\cdots,y_k$, we know that if $x_k=y_i$ then $i=k$, otherwise we could get a monochromatic $S$-difference sequence with length larger than $k$ by concatenating these sequences.. We construct a new finite coloring $\beta: \N \to [1,2r]$ by changing only each \emph{tail} $x_k$ in every maximal monochromatic $S$-difference sequence by $\beta(x_k)=d+\alpha(x_k)$. For this new coloring the length of the maximal monochromatic $S$-sequences is less than $k$. This concludes the proof.
\end{proof}




We say that a set $S \subset \N$ is \emph{density intersective}
if, given any set of integers with positive upper density $A \subset \N$, we have $(A-A) \cap S \neq \emptyset$. The \emph{upper density} of a set $A \subseteq \N$ is defined as $\limsup_{n \to \infty} \frac{A(n)}{n}$ where $A(n) = |\left\{ i \in A : i \leq n \right\}|$. It is a result of K\v{r}\'{i}\v{z} \cite{K} that there are chromatically intersective sets which are not density intersective. See \cite[Theorem 1.2]{MC} for the statement of the theorem using this terminology.
In any case, we have the following analogue of Theorem \ref{chrom-acc}.

\begin{theorem}
Let $S \subset \N$ be a density intersective set and $A \subset \N$ any set with positive upper density. Then, the set $A$ contains arbitrarily long $S$-difference sequences $x_1,x_2,\cdots,x_k$, $x_i \in A$ for $1\leq i \leq k$ and $x_{i+1}-x_i \in S$ for $1\leq i \leq k-1$.
\end{theorem}

\begin{proof}
Suppose there is a set $A \subset \N$ of positive upper density such that the theorem does not hold.
We define the \emph{order in A} of $x \in \N$ as the length of the largest $S$-sequence with $x_1 = x$ and $x_i \in A$ (it can be infinity).
As the theorem does not hold for the set $A$ then the supremum of all orders in $A$ is finite, $k \geq 2$. We suppose this is the minimal $k$ over all such sets $A$.
Let $B$ be the set of elements in $A$ with order equal to $k$. In $B$ every integer has order less than or equal to $1$ and in $A \backslash B$ every integer has order less than or equal to $k-1$. One of these sets has positive upper density and the supremum of all its orders less than $k$ which contradicts the minimality of $k$.
\end{proof}

Recall that a set which contains arbitarily large difference sets, in particular the difference set of an infinite set,
is density intersective and, a fortiori, chromatically intersective by a simple pigeonhole argument.
Combining the results in this section we conclude that the difference sets in Theorem~\ref{construction} are accessible, chromatically and densitive intersective. Jungi\'{c} \cite{J1} has already shown that there exist chromatically intersective sets which are not ladders but his examples are not given explicitly.

\begin{question}
  Is there an intersective set which is not a $2$-ladder?
\end{question}

\section{Infinite monochromatic walks}

In this section we will consider infinite monochromatic $S$-difference sequences which we will call infinite $S$-walks. If we take $S=A-A$ with $A \subset N$ an infinite set, then given any finite coloring there will always exist infinite monochromatic $S$-walks. This is because one of the colors contains an infinite number of elements of the set $A$ whose elements define an infinite monochromatic $S$-walk.

Given a set $S \subset \N$ we define its \emph{order}, ${\rm ord}(S)$, to be the largest positive integer $r$ for which every $r$-coloring contains infinite monochromatic $S$-walks. If no such larger $r$ exists, then we say that ${\rm ord}(S)=\infty$ and that $S \subset \N$ is an \emph{infinitely walkable} set.

\begin{theorem}
If $S=\{n^2:n \in \N\}$ is the set of squares, then ${\rm ord}(S)\geq 2$.
\end{theorem}

\begin{proof}
Suppose there was a $2$-coloring $\alpha: \N \to [1,2]$ avoiding infinite monochromatic $S$-walks. For this coloring there had to exist \emph{dead ends}: an integer $n \in \N$ such that $\alpha(n+k^2) \neq \alpha(n)$, for all $k \geq 1$. This implies that the infinite set $T=\{n+k^2:k \geq 1\}$ is necessarily monochromatic. This set is a translation of the set of squares.  
We will now construct an infinite $S$-walk on the set of squares.

Define $z_n$ by the recursion 
\[  z_1=6, \ \ z_{n+1}  =  \frac{z_n^2}{4} + 1. \]
By induction we see that  the terms of this sequence are integers satisfying $z_n \equiv 2 \bmod 4$.
We claim that the sequence $z_n^2$ is an infinite $S$-walk. Indeed, $z_{n+1}^2 - z_n^2 = \bigl(z_n^2-1\bigr)^2$.
We can now  translate this infinite walk to the set $T$ finishing the proof.
\end{proof}

\begin{theorem}
Let $S=\{s_1,s_2,\cdots\}$ be an infinitely walkable set. Then the following limit
\begin{equation*}
\liminf \limits_{i \rightarrow \infty} (s_{i+k}-s_i)
\end{equation*} is finite.
\end{theorem}

\begin{proof}
We will prove that if $\liminf (s_{i+k}-s_i)=\infty$ then there is a $(2k+2)$-coloring for the positive integers which avoids monochromatic infinite $S$-walks.
Consider the following $2$-coloring of the positive integers: partition the integers into finite intervals and color consecutive intervals with different colors. Moreover, if $l_j$ is the length of the $j$-th interval then $l_j > s_{x_j}$ where $x_j$ is such that $s_{i+k}-s_i > l_1 + \cdots + l_{j-1}$ for all $i \geq x_j$.
We then change this $2$-coloring into a $(2k+2)$-coloring in the following way: if $x$ belongs to the $(j+1)$-th interval there are at most $k$ elements of the first $j$ intervals that have the same color of $x$ in the original $2$-coloring and are the difference between $x$ and an element in $S$. That is because such elements are in the first $j-1$ intervals and the difference between $x$ and one of those elements is at least $l_j > s_{x_j}$. If the smallest of those differences is $s_i > l_j > s_{x_j}$ then $s_{i+k} - s_i > l_1 + \cdots + l_{j-1}$. As $x - s_i$ is in one of the first $j-1$ intervals then $x - s_{i+k} = (x - s_i) - (s_{i+k} - s_i) < (x - s_i) - (l_1 + \cdots + l_{j-1}) < 0$. As we use two sets of $k+1$ colors, one for each color in the original $2$-coloring, it is possible to choose a color for $x$ which is different from the colors of the elements in the first $j$ intervals already mentioned.

For this new coloring, every monochromatic $S$-sequence is contained in one finite interval, therefore is has finite length.
\end{proof}

\begin{corollary}
Let $S_k=\{n^k:n \in \N\}$ be the set of $k$-powers. Then, ${\rm ord}(S_k) \leq 3$.
\end{corollary}

\begin{problem}
What is the order for the set of squares? What is the order for the set of cubes?
\end{problem}

For the squares we know that $2 \leq {\rm ord}(S_2) \leq 3$. For the cubes we know $1 \leq {\rm ord}(S_3) \leq 3$.

\begin{question}
Is there a ladder set $L$ such that ${\rm ord}(L)=1$?
\end{question}

\begin{question}
Is it true that every infinitely walkable set $S \subset \N$ contains an infinite difference set $A-A \subset S$ with $A$ an infinite set of integers?
\end{question}

\section{The distance graph}

Given a set of positive integers $D$, we define the \emph{distance
graph} $G_D$ where $V(G)=\Z$ and $E(G)=\{(x,y) \in V\times V:|x-y|
\in D\}$. If $D \subset \N$ is chromatically intersective, then the associated
distance graph $G_D$ has infinite chromatic number. So Theorem \ref{chrom-acc} can be reformulated as follows:

If  $D \subset \N$ is such that $G_D$ has infinite chromatic number and $\alpha:V \to [1,r]$ is
a finite $r$-coloring of the vertices, then in  $G_D$
there exist arbitrarily long upwards going monochromatic paths.

This can be extended from distance graphs to arbitrary graphs.

\begin{theorem}
Let $G$ be a graph of infinite chromatic number and let  $\alpha:V \to [1,r]$ be
an $r$-coloring of the vertices. Given any ordering of $V$,
there exist arbitrarily long upwards going monochromatic paths.
\end{theorem}

\begin{proof}
Assume the contrary, and let $k$ be the maximal lenght of  upwards going monochromatic paths. 
For a vertex $x$, let $\beta(x)$ be the  maximal lenght of  upwards going monochromatic paths starting from $x$ (with $\beta(x)=0$ if $x$
is not connected to any higher vertex). Put $\gamma(x)=\alpha(x)+d\beta(x)$. Then
 $\gamma:V \to [1,r(1+k)]$ is a coloring in which no two vertices of the same color are connected, a contradiction.
\end{proof}

 \begin{question}
   An ordering of the vertices is the same as making it a directed graph without directed cycles. Is the above result true for arbitrary directed graphs?
 \end{question}

\begin{theorem}
Let $D \subset \N$ and suppose that for any finite coloring of the vertices of $G_D$ there exists monochromatic copies of the finite graphs $G_1$ and $G_2$ in $G_D$. Then for any finite coloring of the vertices of $G_D$ there exists a monochromatic copy of the graph $G_1 \times G_2$.
\end{theorem}

\begin{proof}
Let $\alpha:V \to [1,r]$ be an  $r$-coloring of the vertices of $G_D$. There is a positive integer $N$ such that every subset of $N$ consecutive integers in $V$ contains a monochromatic copy of $G_1$. Define a new coloring $\beta:\N \to [1,r]^N$ via \[ \beta(n+1) = \left(\alpha(nN+1), \alpha(nN+2), \cdots, \alpha(nN+N)\right). \]
Let $D' = \frac{N\N \cap D}{N}$. Note that because $G_D$ has infinite chromatic number, we have that $D' \neq \emptyset$. Then $G_{D'}$ also contains a monochromatic copy of $G_2$ for any finite coloring of its vertices. This is true because we take the product of any finite coloring with a $\bmod N$ coloring then all the differences between elements of the same color are in $N \Z$. In particular, $G_{D'}$ contains a monochromatic copy of $G_2$ for the coloring $\beta$. This corresponds to a monochromatic copy of $G_1 \times G_2$ in $G_D$.
\end{proof}

By combining the previous two theorems we get the following corollary.

\begin{corollary}
Let $D \subset \N$ such that $G_D$ has infinite chromatic number and
a finite $r$-coloring of the vertices $\alpha:V \to [1,r]$.
There exist $n$-dimensional monochromatic grids graphs with arbitrarily large side lengths.
\end{corollary}

\begin{question}
What else can we say about the structure of the distance graph $G_D$,
given that it has infinite chromatic number? In particular, which other
monochromatic subgraphs are forced in all finite colorings?
\end{question}

\section{Acknowledgements}

The authors would like to thank the referees for their helpful comments and suggestions.

\end{document}